\documentclass{article}

\usepackage{amsmath,amsthm,amssymb,amsfonts,hyperref,cite,footnote,cases}

\newtheorem{theorem}{Theorem}[section]
\newtheorem{definition}[theorem]{Definition}
\newtheorem{exa}[theorem]{Example}

\newtheorem{lemma}[theorem]{Lemma}
\newtheorem{conj}[theorem]{Conjecture}
\newtheorem{prop}[theorem]{Proposition}
\newtheorem{corollary}[theorem]{Corollary}

\title{A Simplicial Tutte ``5''-flow Conjecture}
\date{}
\author{Bradley Lewis Burdick$^1$}

\begin{document}

\maketitle
{\let\thefootnote\relax\footnote{\noindent$^1$Department of Mathematics, University of Oregon. Eugene, OR 97403.\\ Email: \url{burdick.28@osu.edu} or \url{bburdick@uoregon.edu}}}

\begin{abstract}
This paper concerns a generalization of nowhere-zero modular $q$-flows from graphs to simplicial complexes of dimension $d$ greater than $1$. A modular $q$-flow of a simplicial complex is an element of the kernel of the $d\text{th}$ boundary map with coefficients in $\mathbb{Z}_q$; it is called nowhere-zero if it is not zero restricted to any of the facets of the complex. Briefly noting connections to other invariants of simplicial complexes, this paper provides a generalization of Tutte's 5-flow conjecture, which claims the universal existence of a $5$-flow for all bridgeless graphs. Once phrased, this paper concludes with bounds on what ``5'' ought to be for simplicial complexes of dimension $d$: proving a lower bound linear in $d$ and a partial upper bound exponential in $d$. 
\end{abstract}

\section{Introduction}

This paper originated from a project seeking to connect the Tutte-Krushkal-Renardy (TKR) Polynomial, which the author studied in \cite{BBC}, to the flow qausipolynomial introduced in \cite{BK} in a way analogous to how the flow polynomial of a graph is a specialization of its Tutte Polynomial. However, the connection is already implicit in \cite{BK}, though the work of \cite{KR} seems to be unknown to the authors of \cite{BK}. Though the TKR polynomial is merely a topologically formulation of the Tutte polynomial of the vectorial matroid of the matrix $\partial:C_d\rightarrow C_{d-1}$, the topological perspective allows one to gather some unique properties for triangulation of manifolds. I will exploit the knowledge gleaned from working on \cite{BBC} to prove some novel facts about the flow quasipolynomial in Section \ref{thetopologicalperspective}. But since these observations are so close to the foundation in \cite{BK} and later work by \cite{BBGM}, they do not merit their own paper, and so I continued exploring the subject.  

Given that the field is currently active in generalizing problems of graph theory to simplicial complexes, the obvious choice is to look at one of the most renowned open conjectures of classical graph theory: Tutte's 5-flow Conjecture. Simply put, Tutte conjectured in 1952 \cite{Tu} that all bridgeless graphs have a $5$-flow. He demonstrated that the Petersen graph has no $4$-flow but indeed has a $5$-flow, and ended his discussion there. His conjecture was not fully vindicated until 1976 and 1980 when every bridgeless graph was exhibited to have an $8$-flow and a $6$ flow respectively by Jaeger \cite{Ja} and Seymour \cite{Se}. The conjecture, however, is still unproven. The generalization is simple: replace ``5'' with some number $\kappa(d)$.\footnote{The notation $\kappa(d)$ is inspired by \cite{Ja}, who uses $\kappa(G)$ to refer to the minimal number for which there is a nowhere zero flow.}

\begin{conj}\label{weak} For every dimension $d$ there exists a number $\kappa(d)<\infty$, such that every bridgeless complex of dimension $d$ has a nowhere-zero $q$-flow for some $q\le\kappa(d)$.
\end{conj}

What exactly ``bridgeless'' ought to mean is worked out in Section \ref{background}, and is finally stated in Definition \ref{whatisabridge}. Of course, this conjecture is uninteresting; Tutte gave a similar conjecture, which he quickly replaced with a stronger conjecture without any justification but exhibiting a lower bound. We follow his lead, and claim $\kappa(d)$ is the obvious linear term.

\begin{conj}\label{strong} $\kappa(d)=d+4$.
\end{conj}

Like Tutte, we construct a counterexample to exhibit a lower bound for $\kappa(d)$. Tutte used the Petersen graph; for simplicity, we settle on justifying the following with the complete complex of dimension $d$ on $d+3$ vertices. We prove the following theorem in Section \ref{thelowerbound}.

\begin{theorem}\label{lower} $\kappa(d)>d+2$.
\end{theorem}

Unlike Tutte, we have the benefit of Jaeger and Seymour in showing us a method for constructing an upper bound on $\kappa(d)$. Seymour's argument, however, sadly relies on the fact that $6$ is composite; in general, $d+5$ will often be prime. Jaeger appeals only to matroid theory to give a construction of an $8$-flow. We follow this method, and settle on a partial upper bound that we prove in Section \ref{theupperbound}.

\begin{theorem}\label{upper} All simplicial complexes of dimension $d$ that are also $(d+2)$-facet-connected have a nowhere-zero $q$ flow for some $q\le2^{d+2}$-flow.
\end{theorem}

These theorem's give definite credence to Conjecture \ref{weak}, and the linear growth exhibited Theorem \ref{lower} and the original intuition of Tutte lends some plausibility to Conjecture \ref{strong}. With Tutte's original conjecture, all remain open. In the paper that follows, we will see that the general case may be quite different. In the very least, this opens up further inquiry into an otherwise hard open problem. 

\section{Background}\label{background}

This section begins with the necessarily pedantic notation of algebraic topology used in this paper. Having sorted that out, we immediately follow with the substantive notions that will form the basis of my arguments. Among these are some basics of matroid theory with special attention to their relationship to simplicial complexes and the desired analogy to graphs. We conclude with more formal definitions and discussions of the modern concepts being explored, i.e. $q$-flows.

Throughout the paper $\Delta$ will be a simplicial complex (or a CW complex if stated) of dimension $d$.  Let $F$ be the of \emph{facets} or $d$-dimensional cells and let $R$ be the set of \emph{ridges} or $(d-1)$-dimensional  cells of $\Delta$. We will assume the reader is familiar with the following constructions from algebraic topology, see \cite{Ha}. Let $\partial:C_d(\Delta)\rightarrow C_{d-1}(\Delta)$ be the $d${th} chain boundary map; this is identified with a linear transformation $\partial: \mathbb{Z}^{|F|}\rightarrow \mathbb{Z}^{|R|}$. When several complexes are being considered at once we will use $\partial(X)$ to denote the map $C_d(X)\rightarrow C_{d-1}(\Delta)$. $H_n(X;G)$ will denote the $n$th reduced\footnote{We need to use reduced homology only to include graphs into our theorems, but if we assume $d>1$ we need not worry about it. I will not specify reduced homology after this, and I reject the convention of using twiddle notation for the sake of aesthetics.} homology group of the space $X$ with coefficients in an Abelian group $G$; if $G$ is omitted it is assumed to be $\mathbb{Z}$. Let $\beta_n(X)$ denote the rank of the $n$th homology group of X, called the \emph{Betti number}. $\mathbb{Z}_q$ denotes the quotient group $\mathbb{Z}/q\mathbb{Z}$ and not the $q$-addic integers. Let $\text{Tor}(G,A)$ denote the first Tor functor, and let $\text{tor}(G)$ denote the torsion subgroup of $G$. Let $\Delta_n$ denote the $n$-skeleton of $\Delta$, we will use $|\Delta_n|$ to denote the number of $n$-simplices in $\Delta$. We will often consider sets of facets $X\subseteq F$, which will be conflated with the complex $X\cup \Delta_{d-1}$. So when we write $H_n(X)$ and $\beta_n(X)$ we mean $H_n(X\cup \Delta_{d-1})$ and $\beta_n(X\cup \Delta_{d-1})$. 

One of the essential combinatorial objects is the matroid, which Whitney introduced as a generalization of graphs. One ought to think of matroids as a generalization of sets of vectors.

\begin{definition} A matroid is a pair of sets $(E,I)$, where $I\subseteq \mathcal{P}(E)$ is a collection of \emph{independent sets} meeting the following requirements:\\ i) $\emptyset\in I$. \\ ii) if $A\subseteq B \in I$, then $A\in I$. \\ iii) If $A,B\in I$ and $A\neq B$, then there is an $a\in A\setminus B$ such that $\{a\}\cup B\in I$.

For any set function $r:\mathcal{P}(E)\rightarrow \mathbb{N}$ we may define $I_r=\{A\subseteq E: r(A)=|A|\}$. Any set function for which $I_r$ satisfies the above axioms will be called a \emph{rank function}, and we may define a matroid $(E,r)=(E,I_r)$.
\end{definition}

 There are several ways to derive a matroid from a graph, but the one Tutte was concerned with is the {cycle matroid} which is equivalent to the {column matroid} of the {incidence matrix} $\partial$ for a graph. Its independent sets are those sets of edges free of a cycle. Though there are analogues of cycles for simplicial complexes, the easiest way we can generalize the matroid to simplicial complexes is simply replacing ``incidence matrix'' with ``boundary map.''

\begin{definition} The simplicial matroid\footnote{A chapter on matroids arising from simplicial complexes in this way was included in \emph{Combinatorial Geometries} \cite{CL}, and has since been revived in a string of papers.} of the complex $\Delta$ is denoted by $M(\Delta)=(F,r)$, where $r(X)=\text{rank}(\partial_X)$, i.e. $M(\Delta)$ is just the vectorial matroid of the columns of $\partial$.
\end{definition}

This rank function can also be written in terms of a complex's Betti numbers, and the proof is simply an application of the rank-nullity theorem.

\begin{prop}[Cardovil \& Lindstr\"{o}m \cite{CL}] For a simplicial matroid, the rank function can be defined as follows.
$$r(X)=|X|-\beta_d(X).$$
\end{prop}

\noindent It should not come as a surprise that the difference in rank of two subcmoplexes $X,Y\subseteq F$ can be expressed in terms of the codimension 1 Betti numbers. This fact was first proven in \cite{KR}, but the proof will be repeated as it derives an important equation (1) that we will need to cite in later proofs.

\begin{prop}[Krushkal \& Renardy \cite{KR}]\label{rankdifference} For a simplicial matroid, the difference of ranks can be written as follows.
$$r(X)-r(Y)=\beta_{d-1}(Y)-\beta_{d-1}(X).$$
\end{prop}
\begin{proof}
We begin by considering the Euler characteristic of $X$ and $Y$. Since $X$ and $Y$ have the same $(d-1)$-skeleton, it is clear from the combinatorial sum that $\chi(Y)-\chi(X)=(-1)^d(|Y|-|X|)$. Since $X$ and $Y$ share a $(d-1)$-skeleton, we also know that the chain groups $C_n(X)$ and $C_n(Y)$ and the homomorphisms $\partial_n(X)$ and $\partial_n(Y)$ are respectively the same for $n<d$. So it must be that $\beta_n(X)=\beta_n(Y)$ for $n<d-1$. It is now clear from the homological sum that $\chi(Y)-\chi(X)=(-1)^d(\beta_d(Y)-\beta_{d-1}(Y)-\beta_d(X)+\beta_{d-1}(X)).$ Thus
\begin{align}|Y|-|X|= \beta_d(Y)-\beta_{d-1}(Y)-\beta_d(X)+\beta_{d-1}(X).\end{align}

Rearranging this equation verifies the claim.
\end{proof}

In matroid theory, to every matroid is associated a \emph{dual matroid}. Though the term ``dual'' is mysterious in the abstraction of matroids, the construction is rather simple, and was defined to mirror the planar dual of a graph and the dual of a vector space.

\begin{definition} The dual of a matroid $M$ is denoted $M^*$, and is defined as $(F,r^*)$, where $r^*$ is called the corank and is defined $r^*(X)=|X|+r(F\setminus X)-r(F)$.
\end{definition}
We will need to consider the dual of simplicial matroids in the rather technical proof of Theorem \ref{theupperbound}. Proposition \ref{rankdifference} allows us to write the corank function of the simplicial matroid in simpler terms.

\begin{corollary} For a simplicial matroid the corank function may be written as follows.
$$r^*(X)=|X|+\beta_{d-1}(\Delta)-\beta_{d-1}(\Delta\setminus X).$$
\end{corollary}

As mentioned, one ought to think of matroids as generalized sets of vectors. The terminology \emph{independent}, \emph{rank}, and \emph{dual} reflects this. The term \emph{base} is also used to refer to a maximal independent sets. The matroid's historical relationship with graphs has led to a number of graph flavored terms: \emph{loop}, \emph{circuit}, and \emph{cycle}, all of which may be prefixed with \emph{co-} to indicate the dual notion. A \emph{loop} is a single element dependent set, a \emph{circuit} is a minimal dependent set, and a \emph{cycle} is any dependent set. The graph term we are perhaps most interested in is the \emph{bridge}. In a graph, a bridge is a single element separating edge cut, i.e. deleting it results in two disconnected subgraphs. Indeed in matroids, a bridge is a single element separating set (\cite{Ox} Ch 4). But properly a bridge is a coloop, so it is independent, contained in no circuits, and contained in every base.

The obvious choice for the definition for bridge of a simplicial complex is to coincide with bridges of the simplicial matroid. For graphs, a bridge when removed creates a new connected component, or in terms of homology, its boundary constitutes a unique a codimension one $\mathbb{Q}$-cycle. So if we are to use homology as a basis for analogy, a bridge for an arbitrary complex ought to do the same. In \cite{BBC}, we found this to be the necessary definition for the TKR polynomial to satisfy the desired contraction-deletion property. It turns out that the definitions for bridge all coincide, and can be defined as follow.

\begin{definition} A bridge is a facet $f\in F$ that is a coloop of the matroid $M(\Delta)$, i.e. $r^*(f)=1+\beta_{d-1}(\Delta\setminus\{f\})-\beta_{d-1}(\Delta)=0$. In other words removing $f$ destroys a boundary element, $\beta_{d-1}(\Delta\setminus \{f\})=\beta_{d-1}(\Delta)+1$.
\end{definition}

In graphs, a bridge can also be called an \emph{edge cut}, and graphs that have a bridge are called \emph{1-edge-connected} because it only takes an edge cut to separate them. There is of course a notion of $k$-edge-cut and $k$-edge-connected, which we will phrase in the general simplicial complex sense as follows.

\begin{definition} A $k$-facet-cut is a set $X\subseteq F$ with $|X|=k$ such that $\beta_{d-1}(\Delta\setminus X)>\beta_{d-1}(\Delta)$. A complex is $k$-facet-connected if it has no $l$-cuts for $l<k$. 
\end{definition} 

It should be absolutely clear what bridgeless ought to mean at this point, but for sake of completion we phrase the following definition.

\begin{definition}\label{whatisabridge} A simplicial complex is bridgeless if none of its facets are bridges, or in other words is $2$-facet-connected.
\end{definition}

One might ask the relationship with the established notion of $k$-connectedness of $M(\Delta)$  (\cite{Ox} Ch 8). An $k$-edge-connected graph does necessarily have a $k$-connected matroid, but there are characterizations of $k$-connected graphic matroids in terms of the graphs. Similar results can be stated for simplicial complexes.

As mentioned, one of the interesting properties of bridges, is that they are contained in every basis of a matroid, thus they are independent from any other set. From this we can see that that a bridge corresponds to a column vector of $\partial$ independent of any other set of column vectors. Thus any $q$-flow must evaluate to $0$ on a bridge, preventing any complex with a bridge from having a nowhere-zero $q$-flow. We may ask the converse: do all bridgeless complexes have a nonzero flow quasipolynomial? This amounts to asking if a matrix with no independent columns have a nontrivial kernel $\mod q$, which is true. But we phrase this as a proposition, and provide a constructive proof as a corollary to Lemma \ref{constructedflow}.

\begin{prop}\label{bridgelessflow} If a simplicial complex is bridgeless then it has a nowhere-zero $q$-flow for some $q$.
\end{prop}

The converse, while true for graphs, is actually false for all other dimensions. This is simply because there exists spherical maps of degree greater than one. So there may a kernel mod $q$ while even while everything remains independent over $\mathbb{Z}$.

We mentioned a base of matroid, a base is a maximal independent set, i.e. any set satisfying $|X|=r(X)=r(F)$. In graph theory, the bases of the cycle matroid for a connected graph are called \emph{spanning trees} and in general are called \emph{maximal spanning forests}. The term \emph{simplicial spanning trees} or \emph{cellular spanning trees} can be found in \cite{DKM1} and \cite{DKM2} respectively for simplicial or CW complexes satisfying the assumption $\beta_{d-1}(\Delta)=0$, which is a generalization of the assumption of connectedness.\footnote{The study of graphs thankfully reduces to the study of connected graphs, but for higher dimensions, the study of complexes does not reduce to complexes satisfying $\beta_{d-1}(\Delta)=0$.}  Dropping this assumption, \cite{Pe} introduced the term \emph{$k$-tree}. While bases are not the primary interest of this paper, they are necessary to the proof of Theorem \ref{upper}, so we will unify the terminology in a way consistent with the analogy that $\beta_{d-1}(X)$ corresponds to the number connected components of a graph.

\begin{definition} Let $T\subseteq F$. We call $T$ a \emph{forest} if $\beta_d(T)=0$, it is \emph{maximal} if $\beta_{d-1}(T)=\beta_{d-1}(\Delta)$. Additionally if $\beta_{d-1}(T)=0$ it is called a \emph{tree}, and if $\beta_{d-1}(T)=\beta_{d-1}(\Delta)=0$ it is a \emph{spanning tree}.
\end{definition}

Regardless of what we call these subcomplexes, their importance is in their correspondence to the underlying matroid.
\begin{prop}[Duval et al. \cite{DKM1,DKM2} Petersson \cite{Pe}] The maximal forests or spanning trees of $\Delta$ are the bases of $M(\Delta)$.
\end{prop}

As the name of this paper suggests, Tutte and his conjecture in \cite{Tu} are of chief importance. For those not familiar with classic graph theory we will briefly discuss the concepts at hand. A flow is an edge weight of a directed graph that preserves mass around vertices, i.e. where the inflowing weights equal the outflowing. There were three types of flows that Tutte considered: modular $q$-flows, integral $q$-flows, and $G$-flows. Both modular and integral $q$-flows assign weights from $\{0,\dots,q-1\}$, but the equivalence is carried out modulo $q$ or in the integers respectively. $G$-flows assign weights from an Abelian group $G$ where the equivalence is carried out in $G$. Tutte showed that the number of nowhere-zero $G$-flows, modular $q$-flows, and integral $q$-flows all equal $\Phi_\Delta(q)$ whenever $|G|=q$, where $\Phi_\Delta(q)$ is a polynomial in $q$. To do this he introduced his dichromate polynomial, which in modern terms is the Tutte polynomial of the cycle matroid. 

\begin{definition}The Tutte polynomial of a matroid $M=(F,r)$ is 
$$T_M(x,y)=\sum_{X\subseteq F} (x-1)^{r(F)-r(X)}(y-1)^{|X|-r(X)}.$$
\end{definition}

We will see how this works in relationship to $\Phi_\Delta(q)$ in Proposition \ref{graphspecializations}. But first let us define our desired generalization of $q$-flows to simplicial complexes. We will do this in a way completely analogous, which amounts to replacing ``edges'' with ``facets'' and ``vertices'' with ``ridges.'' 

\begin{definition} $\varphi\in \{0,\dots, q-1\}^{|F|}$ is called a $q$-flow of $\Delta$ if $\partial\varphi^T\equiv 0\mod q$. It is nowhere-zero if $\varphi\in\{1,\dots,q-1\}^{|F|}$. The number of nowhere-zero $q$-flows is denoted $\Phi_\Delta(q)$ and is called the \emph{flow quasipolynomial.}
\end{definition}

Another central concept of classic graph theory is vertex coloring. As the name suggests, every vertex of a graph is colored with a number from $\{0,\dots,k-1\}$. A coloring is \emph{proper} if no two adjacent vertices have the same color. Proper graph colorings were introduced to study an old problem of cartography, where no two bordering countries could be colored the same way. Again the number of proper vertex colorings is a polynomial in the number of colors $X_\Delta(k)$. Vertex colorings are not the main focus of this paper, but since an understanding of what they are is essential to Section \ref{thetopologicalperspective}, we will phrase the generalization to simplicial complexes. The definition amounts to replacing ``vertices'' with ``ridges,'' and ``adjacent'' with ``borders the same facet.'' 

\begin{definition} $\chi\in\{0,\dots, k-1\}^{|R|}$ is called a $k$-coloring of $\Delta$. It is proper if $\chi\partial\cdot e_j\not\equiv 0\mod k$ for all standard basis vectors $e_j\in \mathbb{Z}^{|F|}$. The number of proper $k$-colorings is denoted $X_\Delta(k)$ and is called the \emph{chromatic quasipolynomial.}
\end{definition}

It should be mentioned that a quasipolynomial is a polynomial with coefficients that are integer valued periodic functions of the polynomial's argument. The fact that these counting functions are quasipolynomials is not trivial, and is demonstrated in \cite{BK}. Let's give a quick example of an instance when $\Phi_\Delta(q)$ is not a polynomial.

\begin{exa} Let $G\cong \bigoplus \mathbb{Z}_{q_i}^{r_i}$, where $\{q_i\}$ is a sequence of distinct primes, and let $\Delta$ be a triangulation of the Moore space with $H_{d-1}(\Delta)\cong G$ (see \cite{Ha}), then $\Phi_\Delta(q)$ has the following form.

$$ \Phi_\Delta(q) = \prod [\gcd(q_i,q)-1]^{r_i}.\footnote{This is perhaps not obvious, but we will see in Theorem \ref{combinatorialinvariant} that $\Phi_\Delta(q)$ is invariant under refinement. So it suffices to think about the simplest CW structure of the Moore space consisting of disjoint complexes comprised of a $d$-cell, a $(d-1)$-cell, and a 0-cell.}$$
\end{exa}

\noindent The easiest concrete instance of this example is $\mathbb{R}P^2$, which is the Moore space for $\mathbb{Z}_2$. Explicitly if $\Delta$ is a triangulation of $\mathbb{R}P^2$, then $\Phi_\Delta(q)$ is as follows. Having infinitely many roots, it cannot be a polynomial.

$$\Phi_\Delta(q)=\begin{cases}1 &\text{if }q\text{ is even}\\ 0 &\text{if }q\text{ is odd.} \end{cases}$$

You might ask why we do not consider the varieties of flows other than modular. As mentioned, Tutte proved the bijection of flows via a contraction-deletion property of his dichromate polynomial. For simplicial complexes, flows and colorings are not counted by polynomials, so they are not immediately subject to the same contraction-deletion property. It turns out that they are beholden to a similar property first discovered in \cite{BBGM}, and integral flows and colorings get attention in both \cite{Go} and \cite{BBGM}. $G$-flows get less attention, but we will actually use them in an argument in Section \ref{theupperbound}. Most importantly, the number of modular, integral, and $G$ flows are not the same for a fixed order $q$. To see this, consider $\Delta \cong \mathbb{R}P^2 \cup \mathbb{R}P^2$. To say it has an integral $4$-flow amounts to providing an orientation, which is impossible, so there are no integral $4$-flows. The only modular $4$-flow is to assign the value two to every simplex. But there are three nonzero elements of exponent two of $V_4$, so there are nine nowhere-zero $V_4$-flows. 

We will work with abstract simplicial complexes for the bulk of the paper, but it comes in handy to be familiar with a large class of explicit simplicial complexes. The easiest class to describe are the complete simplicial complexes. The complete simplicial complex of dimension $d$ on $n$-vertices is denoted $K_n^{d+1}$ and is exactly as it sounds: the complex comprised of every distinct $d$-simplex obtainable from $n$ distinct vertices. $K_n^2$ is usually just denoted $K_n$, and is the complete graph on $n$ vertices. $K_n^{n}$ is often denoted $\Delta^{n-1}$, and is the $(n-1)$-simplex. 

\begin{definition} For $d<n$ let $K_n^{d+1}$ denote the complete simplicial complex of dimension $d$ on $n$ vertices, i.e.
$$K_n^{d+1}=\{X\subseteq\{1,\dots,n\}:|X|\le d+1\}.$$
\end{definition}

I may state facts of graph theory, matroid theory, and algebraic topology without reference. For further exposition on any subject or for any exposition I have omitted, consult Wikipedia, but \cite{Di}, \cite{Ox} and \cite{Ha} are standard references for the respective fields.

\section{The Topological Perspective}\label{thetopologicalperspective}

Whereas graphs are largely studied as combinatorial objects, they are as easily viewed as topological spaces. Simplicial complexes were historically used as a combinatorial approach to topology, so they are studied in both perspectives. This section begins by phrasing a topologically defined version of the Tutte polynomial for a simplicial matroid. We will then trace the study of it in \cite{BBC}, \cite{BK}, \cite{Go}, and \cite{KR} in analogy to the Tutte polynomial of a graph to motivate two theorems about $\Phi_\Delta(q)$.

\begin{definition} The Tutte-Krushkal-Renardy (henceforth TKR) polynomial\footnote{\cite{KR} defined their polynomial with $x$ and $y$ monomials rather than $(x-1)$ and $(y-1)$. This constitutes a change of variables from standard definitions of the Tutte polynomial, so we revert back.} of a complex $\Delta$ is denoted $T_\Delta(x,y)$ and defined as follows.
 $$T_\Delta(x,y) = \sum_{X\subseteq F} (x-1)^{\beta_{d-1}(X)-\beta_{d-1}(\Delta)}(y-1)^{\beta_d(X)}.$$
\end{definition}

The main purpose of \cite{KR} is to give further credence to this definition as being the Tutte polynomial of a complex. This was partially realized with the following result that generalizes the fact for graphs and graphic matroids.

\begin{prop}[Krushkal and Renardy \cite{KR}] $T_\Delta(x,y)=T_{M(\Delta)}(x,y)$.
\end{prop}

As we have mentioned, Tutte introduced his graph polynomial to study flows and colorings. We will see shortly, how he did this, but it seems that the TKR polynomial has also been exploited in the study of $\Phi_\Delta(q)$ and $X_\Delta(k)$.

\begin{prop}[Beck \& Kemper \cite{BK}, Godkin \cite{Go}\footnote{\cite{BK} proved the first equality, and \cite{Go} the latter. Though neither used the TKR polynomial, instead using $T_{M(\Delta)}(x,y)$. And neither of them phrased their results as I have. Rather, \cite{BK} said equality holds for infinitely many values of $q$, and \cite{Go} gave certain criteria on $\Delta$ for equality to hold. Regardless, a proof of these facts will come as a corollary to the slightly more general Lemma \ref{specializations}}]\label{flowspecial}
$\Phi_\Delta(q)$ is a polynomial if and only if
$$\Phi_\Delta(q)=(-1)^{\beta_d(\Delta)}T_{\Delta}(0,1-q).$$
$X_\Delta(k)$ is a polynomial if and only if 
$$X_\Delta(k)=(-1)^{|F|-\beta_d(\Delta)}k^{|R|-|F|+\beta_d(\Delta)}T_\Delta(1-k,0).$$
\end{prop}

This perspective of $\Phi_\Delta(q)$ allows us to import some of the ideas from \cite{BBC}. Consider first an old polynomial invariant introduced by Roul Bott in 1952; he defined his ``R'' polynomial as a generalization of the flow polynomial. He proved that it was a combinatorial invariant much in the same way as the flow polynomial is, but he gave no sense of what a flow was that it measured. The connection to flows was not justified until Wang proved for graphs that it was the chromatic polynomial of the planar dual, and showed it satisfied the same contraction-deletion relations.

\begin{definition}[Bott \cite{Bo}, Wang \cite{Wa}] Bott's ``R'' polynomial is defined to be
$$R_{\Delta}(\lambda)=\sum_{X\subseteq F} (-1)^{|X|}\lambda^{\beta_d(X)}.$$
\end{definition}

\noindent After Wang, there was still no sense of it counting flows. Until in \cite{BBC} when we established that it equals the same specialization of $T_\Delta(x,y)$ as in Corollary \ref{flowspecial}.

\begin{prop}[Bajo et al. \cite{BBC}] 
$$R_\Delta(\lambda)=(-1)^{\beta_d(\Delta)}T_\Delta(0,1-\lambda).$$
\end{prop}

\noindent Now we can see that 

\begin{corollary} If $\Phi_\Delta(q)$ is a polynomial, then $\Phi_\Delta(q)=R_\Delta(q)$ and is a combinatorial invariant. 
\end{corollary}

\noindent This motivates the following theorem. The proof of which is essentially trivial, but the motivation makes it interesting enough.

\begin{theorem}\label{combinatorialinvariant} $\Phi_\Delta(q)$ is a combinatorial invariant.
\end{theorem}

As promised, we now trace the relationship of flows, colorings, and the Tutte polynomial. The most basic fact connecting the three is the one that gave rise to the name ``dichromate.'' It amounts to the fact that the Tutte polynomial is a generalization of both polynomials.

\begin{prop}[Tutte \cite{Tu}]\label{graphspecializations}
\begin{align*}
 \Phi_G(q) & = (-1)^{|F|+|R|+\beta_{d-1}(G)}T_G(0,1-q)\\
X_G(k) &= (-1)^{|R|-\beta_{d-1}(G)}k^{\beta_{d-1}(G)}T_G(1-k,0).
\end{align*}
\end{prop}

To continue, we restrict our attention to planar graphs. The main reason for doing this, is that it gives rise to at the existence of a planar dual: a planar graph with a unique edge intersecting the original edges connected vertices unique to each region. We will denote the dual of a graph $G$ as $G^*$. Tutte was interested in showing that the flows and proper vertex colorings were related between planar duals, and so he defined his dichromate polynomial. After showing that his polynomial measured both, connecting the flows and colorings of planar duals amounted to checking how Tutte's polynomial changed. It turns out that it merely swaps the variables.

\begin{prop}[Tutte \cite{Tu}]  
$$T_G(x,y) = T_{G^*}(y,x)$$
\end{prop}

This allowed him to conclude the following.

\begin{prop}[Tutte \cite{Tu}] For connected planar graphs $G$ and $G^*$.
$$q\Phi_G(q) = X_{G^*}(q).$$
\end{prop}

First we notice that the equations of Proposition \ref{flowspecial} are nearly identical to the ones in Proposition \ref{graphspecializations}. Indeed, for graphs, the equations are equivalent. This leads us to consider ``planar'' simplicial complexes. Of course, abstract simplicial complexes cannot be planar, but they can embed inside spheres. Indeed, a good analogue of planar dual is spherical duality, so henceforth we will consider $\Delta\cong S^d$.

For graphs, defining planar dual is done easily by drawing out the diagram, placing a dot in each region, and connecting them if the regions share an edge. Doing this for arbitrary simplicial complexes is not so easily done, but we may at least define what we mean by duals.

\begin{definition} $\Delta$ and $\Delta^*$ are said to be dual CW complexes if they both embed via a homeomorphism in the same space $X$ in a way so that for every $n$-cell of $\Delta$ there is a $(d-n)$-cell of $\Delta^*$ that intersect at a point. Given $\Delta$ a simplicial complex, there is an established method to construct its dual $\Delta^*$ via barycentric subdivision, see \cite{Ha}.
\end{definition}

Now that we have a sense for duality, and a generalized dichromate, we may ask if $T_\Delta(x,y)$ satisfies the same ``planar'' duality as Tutte's did. One of the main results of \cite{KR} is that the polynomial is dual in complimentary dimensions for spheres.

\begin{prop}[Krushkal \& Renardy \cite{KR}] If $\Delta$ and $\Delta^*$ are dual cellulations of $S^d$, then $$T_{\Delta_n}(x,y)=T_{\Delta^*_{d-n}}(y,x).$$
\end{prop}

All this suggests us to pose the following theorem.

\begin{theorem}\label{planardual} If $\Delta$ and $\Delta^*$ are dual cellulations of $S^d$, then 
$$\pm q^c \Phi_{\Delta_n}(q)=X_{\Delta^*_{d-n}}(q).$$
\end{theorem}

\subsection{Combinatorial Invariance}

A combinatorial invariant of a simplicial complex is any value that is invariant for combinatorially equivalent siplicial complexes. 

\begin{definition} A subdivision of a simplex $\sigma$ is a homeomorphic complex made by adding an additional vertex: $\sigma' = \{ f\setminus \{i \} \cup \{v\}: i\in \sigma\}$. A refinement $\Delta'$ of $\Delta$ is a complex obtained from any sequence of subdivisions of $\Delta$. Two complexes $\Delta$ and $\Upsilon$ are combinatorial equivalent if there is a common refinement, i.e. so that they have refinements $\Delta'$ and $\Upsilon'$ that are the same simplicial complex $\Delta'=\Upsilon'$.
\end{definition}

This equivalency class is of historic importance due to the long unanswered \emph{Hauptvermutung}, which claimed that all triangulations of a topological space were combinatorially equivalent. If it were true, every combinatorial invariant gives rise to a topological invariant. It is true in dimensions less than 4, but false otherwise. Indeed, not all topological spaces need be homeomorphic to a simplicial complex, so the interest in combinatorial invariants has diminished with the years (see \cite{Ra}).

But if we are to be generalized graph theorists, we should be interested in combinatorial invariants as combinatorial equivalence is known as ``graph homeomorphism'' in dimension 1. Though the motivation for seeking this proof comes from a renowned topologist, the proof is very straightforward.

\begin{proof}[Proof of Theorem \ref{combinatorialinvariant}] It suffices to show that $\Phi_\Delta(q)=\Phi_{\Delta'}(q)$, where $\Delta'$ is the complex obtained by replacing $f\in F$ with it's subdivision $f'$. So let us turn our attention to the complex $f'$. Notice that $f'$ is still homeomorphic to the $d$-ball and has boundary homeomorphic to $S^{d-1}$, though it has quite a few more simplices than the $d$-simplex. Notice that in the subdivision two facets share exactly one ridge in common, and these are the only facets sharing said ridge since it is a $d$-manifold. Thus if we assign $\mathbb{Z}_q$ values to the facets of $f'$ in a way so that they cancel along shared ridges we see that we must assign the same value to each facet of $f'$. This reduces the system of equations to the same as for $\Delta$, and we have a bijection of $q$-flows.
\end{proof}
One may also ask what happens to the simplicial matroid under refinement. Simply, it is a specific matroid extension. The inverse of these extensions are deletions and contractions in matroid theory, which corresponds to deletion and deformation retracts in the CW complexes. From this, one may conclude that every $\mathbb{Q}$-representable matroid exists as the minor of a simplicial matroid. This only further shows the disparity between the 1-dimensional case and the general case. 

This paper works under the assumption that simplicial complexes are the proper object of study rather than general CW complexes. This theorem allows questions of flows on CW complexes to be reduce to simplicial complexes, when the CW complex is triangulable. Since every graph under slight refinement is a simplicial complex, this seems like a fair assumption. 

\subsection{``Planar'' Duality}

One may say that the entire motivation for Tutte to discover his polynomial was to prove the chromatic-flow duality. We will follow his lead and use the TKR polynomial to prove the analogous theorem. Of course, the flow quasipolynomial is only a specialization of the TKR polynomial if it is a polynomial, which we have already seen is not always the case. There is, however, a way to salvage the situation. Via an inclusion-exclusion argument, Godkin \cite{Go} gave a power set expansion of both the flow and chromatic polynomial.

\begin{prop}[Godkin \cite{Go}, Beck et al. \cite{BBGM}]\label{sumform}
\begin{align*}
\Phi_\Delta(q) &=\sum_{X\subseteq F} |\text{\emph{ker}}\partial_X\mod q|(-1)^{|F|-|X|}\\
X_\Delta(k) &=\sum_{X\subseteq F}|\text{\emph{ker}}\partial_X^*\mod k|(-1)^{|X|}
\end{align*}
\end{prop}

While Godkin's argument was firmly rooted in the module theory, relying on the Smith normal form, it may of course be framed in the topological perspective by a bit of homological algebra.

\begin{definition} The $q$-TKR polynomial is denoted $T_\Delta^q(x,y)$ and is defined as follows.
$$T_\Delta^q(x,y)=\sum_{X\subseteq F} t_q(X) (x-1)^{\beta_{d-1}(X)-\beta_{d-1}(\Delta)}(y-1)^{\beta_d(X)}.$$
Where $t_q(X) = |\text{\emph{Tor}}(H_{d-1}(X),\mathbb{Z}_q)|$.
\end{definition}

This may seem contrived, and indeed it is. However, an equivalent definition is given in \cite{BBGM} with a cohomology weighted coefficient. In \cite{BBC},  $|\text{tor}(H_{d-1}(X))|^2$ appears as the coefficient so that the polynomial evaluated at (1,1) yields the weighted count of simplicial spanning trees \`{a} la \cite{DKM1}. And it is even equivalent to the ``prototypical arithmetic Tutte polynomial'' associated to the set $\{\partial f:f\in F\}$ inside the group $C_{d-1}(\Delta)$ introduced by D'adderio and Moci \cite{DM}. Indeed, the introduction of codimension 1 torsion to achieve results analogous to graph theory is found throughout the literature: Kalai \cite{Ka}; Krushkal and Renardy \cite{KR}; Petersson \cite{Pe}; Duval et al. \cite{DKM1, DKM2, DKM3}; and Bajo et al. \cite{BBC}. Regardless, it is a means to the end that the flow and chromatic quasipolynomials may be obtained as the specialization of Tutte quasipolynomial in a way completely analogous to graphs. 

\begin{lemma}\label{specializations} 
\begin{align*}
\Phi_\Delta(q)&=(-1)^{\beta_d(\Delta)}T_\Delta^q(0,1-q)\\
X_\Delta(k)&=(-1)^{|F|-\beta_d(\Delta)}k^{|R|-|F|+\beta_d(\Delta)}T_\Delta^k(1-k,0)
\end{align*}
\end{lemma}

\begin{proof}[Proof of Lemma]  
We will use the form of $\Phi_\Delta(q)$ presented in Proposition \ref{sumform}. First note that $\text{ker}\partial_X\mod q$ is by definition just $H_{d}(X;\mathbb{Z}_q)$. An application of the universal coefficient theorem yields the following.
$$H_d(X;\mathbb{Z}_q)= \mathbb{Z}_q^{\beta_d(X)}\oplus \text{Tor}(H_{d-1}(X),\mathbb{Z}_q).$$

\noindent Thus $|\text{ker} \partial_X\mod q|=q^{\beta_d(X)}t_q(X)$. 

This now allows us to place $\Phi_\Delta(q)$ in the following form.
$$\Phi_\Delta(q)=\sum_{X\subseteq F}t_q(X) q^{\beta_d(X)}(-1)^{|F|-|X|}.$$
Now we verify the claim by comparing this sum to the following.
$$(-1)^{\beta_d(\Delta)} T_\Delta^q(0,1-q) = \sum_{X\subseteq F }t_q(X)(-1)^{\beta_d(\Delta)+\beta_d(X)+\beta_{d-1}(X)-\beta_{d-1}(\Delta)} q^{\beta_d(X)}.$$

\noindent Formula (1) demonstrates that the exponent of negative one in each sum has the same parity. Otherwise the sums are the same, thus 
$$\Phi_\Delta(q)=(-1)^{\beta_d(\Delta)}T_\Delta^q(0,1-q).$$

We will likewise use the form of $X_\Delta(k)$ in Proposition \ref{sumform}. Similar to before, we have the following.
 $$|\ker \partial_X^* \mod k| = k^{|R|-|X|+\beta_d(X)}t_k(X).$$
 
 This now allows us to place $X_\Delta(k)$ in the following form.
$$X_\Delta(k) = \sum_{X\subseteq F} t_k(X) k^{|R|-|X|+\beta_d(X)} (-1)^{|X|}.$$
Finally, we verify the claim by comparing this sum to the following.
\begin{align*}
&(-1)^{|F|-\beta_d(\Delta)}k^{|R| - |F| +\beta_d(\Delta)}T_\Delta^k(1-k,0)=\\
& \sum_{X\subseteq F} t_k(X)(k)^{|R| - |F| + \beta_d(\Delta)+ \beta_{d-1}(X)-\beta_{d-1}(\Delta)}(-1)^{|F|-\beta_d(\Delta)+\beta_d(X)+\beta_{d-1}(X)-\beta_{d-1}(\Delta)}.
 \end{align*}

\noindent Careful consideration of formula (1) shows that all exponents concerned are equal.
\end{proof}

With both flow and chromatic quasipolynomials as a specialization, we need only show that the $q$-TKR polynomial satisfies spherical duality. The proof of this appears almost entirely in \cite{BBC}, but we will show it again for completion. 

The crux of the proof of spherical duality, is a Lemma proved in \cite{KR} that gives a consistent identifications for subcomplexes of $\Delta_n$ and $\Delta^*_{d-n}$.

\begin{lemma}[Krushkal \& Renardy \cite{KR}]\label{morse} If $\Delta$ and $\Delta^*$ are dual cellulations of $S^d$ and $X\subseteq \Delta$ is a subcomplex, then $X$ is homotopy equivalent to $S^d\setminus X^*$ where $X^*$ is the subcomplex of $\Delta^*$ formed of cells which do not intersect $X$.
\end{lemma}

With the association of $X$ to $X^*$ we can easily show that the $q$-TKR polynomial satisfies the following.

\begin{lemma}\label{planarduality} If $\Delta$ and $\Delta^*$ are dual cellulations of $S^d$, then
$$T_{\Delta_n}^q(x,y)=T_{\Delta_{d-n}^*}^q(y,x).\footnote{As mentioned, one needs to consider reduced homology to include graphs into the theory, and this is the only instance where one needs to consider graphs in concurrence with a higher dimension complex.}$$
\end{lemma}

\begin{proof}[Proof of Lemma]  We invoke the universal coefficient theorem, Alexander Duality (see \cite{Ha}), and Lemma \ref{morse} to get the following isomorphisms.
$$H_n(X)/\text{tor}(H_n(X))\oplus\text{tor}(H_{n-1}(X))\cong H^n(X)\cong H_{d-n-1}(S^d\setminus X) \cong H_{d-n-1}(X^*).$$
From this we conclude the following identities.
\begin{align*}
\beta_n(X) & =\beta_{d-n-1}(X^*)\\
\beta_{n-1}(X) & =\beta_{d-n}(X^*)\\
\text{tor}(H_{n-1}(X)) & \cong \text{tor}(H_{d-n-1}(X^*). 
\end{align*}

From the third equality we concluded that $t_q(X)=t_q(X^*)$. Then we note since $\Delta$ and $\Delta^*$ are spheres that $\beta_{n-1}(\Delta_n)=\beta_{d-n-1}(\Delta_{d-n}^*)=0$. We conclude by associating every $X\subseteq \Delta_n$ to $X^*\subseteq \Delta_{d-n}^*$ and compare the summands of $T_{\Delta_n}(x,y)$ to $T_{\Delta^*_{d-n}}(x,y)$.
$$t_q(X)(x-1)^{\beta_{n-1}(X)-\beta_{n-1}(\Delta)}(y-1)^{\beta_d(X)}=$$
$$t_q(X^*)(x-1)^{\beta_{d-n}(X^*)}(y-1)^{\beta_{d-n-1}(X^*)-\beta_{d-n-1}(\Delta_{d-n}^*)}.$$
Careful consideration of the above equalities show that these are equivalent.
\end{proof}

\begin{proof}[Proof of Theorem \ref{planardual}] From Lemmas \ref{planarduality} and \ref{specializations}, we clearly have 
$$(-1)^\varepsilon k^c \Phi_{\Delta_n}(q)=X_{\Delta_{d-n}^*}(q).$$

Where $\varepsilon = |\Delta_{d-n}^*|-\beta_{d-n}(\Delta_{d-n}^*)-\beta_n(\Delta_n)$ and $c=|\Delta_{d-n-1}^*|-|\Delta_{d-n}^*|+\beta_{d-n}(\Delta_{d-n}^*)$.
\end{proof} 

We conclude by introducing a generalization of a third variety of graph coloring: tensions. Like flows, it is an edge weight, but instead of summing to zero around every vertex it sums to zero around every circuit. Every tension can be derived from a proper vertex coloring, in such a way that the tension polynomial is the largest nontrivial divisor of the chromatic polynomial. There is of course a generalization to simplicial complexes, which is studied in detail in \cite{Go} and \cite{BBGM}, that we will define in the following corollary. This allows us to phrase a neater form of spherical duality.

\begin{corollary} A facet weighting from $\{0,\dots,k-1\}$ that sums to zero modulo $k$ along every circuit of $\Delta$ is called a $k$-tension denote it $C_\Delta(k)$. From \cite{BBGM}, it is a fact that
$$t_k(\Delta)C_\Delta(k) = k^{|F|-\beta_d(\Delta)-|R|}X_\Delta(k).$$
Thus $C_\Delta(k)$, trivially satisfies the following.
$$t_k(\Delta)C_\Delta(k)=(-1)^{|F|-\beta_d(\Delta)} T_\Delta^k(1-k,0).$$
And so for dual cellulations $\Delta$ and $\Delta^*$ of a sphere we have:
$$\Phi_{\Delta_n}(q)=(-1)^{\varepsilon}C_{\Delta_{d-n}^*}(q).$$
Where $\varepsilon=|\Delta_{d-n}^*|-\beta_{d-n}(\Delta_{d-n}^*)-\beta_n(\Delta_n)$.
\end{corollary}
\section{The Lower Bound}\label{thelowerbound}
 
 You may have noticed that from the beginning I have assumed that the existence of $q$-flows depends on the dimension of the complex. This section will demonstrate this as a truth constructing a sequence of simplicial complexes that demonstrates that $\kappa(d)$ grows at least linearly in $d$. But first we stop to demonstrate that $\kappa(d)$ is weakly increasing. 
 
 \begin{prop}\label{weakincrease} In the extended reals, $\kappa(d)\le\kappa(d+1)$. 
 \end{prop}
 
In topology there are a few fundamental operations to generate new spaces from existing spaces, and this may be done in a way consistent with a simplicial structure. Among these are the various products as well as cones and suspensions. It was shown in \cite{BK} that if $C\Delta$ is the cone of a complex $\Delta$ that $\Phi_{C\Delta}(q)=0$. We endeavor to count $\Phi_{\Sigma\Delta}(q)$ where $\Sigma\Delta$ is the suspension. 
 
 \begin{lemma}\footnote{To comprehend the proof of the following several statements, one should remember that for $f\in F$, $\partial f$ is defined as the alternating sum of subsets of $f$ obtained by deleting a single element.}  $$\Phi_\Delta(q)=\Phi_{\Sigma\Delta}(q).$$
 \end{lemma}
 \begin{proof}[Proof of Lemma] First note that $\Delta$ and $\Sigma\Delta$ have dimensions $d$ and $d+1$ respectively. One may define $\Sigma\Delta$ by its $n$-skeleton: $\Sigma\Delta_n=\Delta_n\sqcup T_n\sqcup B_n$, where $T_n=\{\{t\}\cup X:X\in\Delta_{n-1}\}$ and $B_n=\{X\cup \{b\}:X\in\Delta_{n-1}\}$. Since this is a disjoint union for each $n$ and $\Delta_{d+1}=\emptyset$ we may write the boundary matrix in the following form.
 
 $$\partial(\Sigma\Delta)=\begin{pmatrix} \partial(T_{d+1},T_d) & \partial(B_{d+1},T_d)\\ \partial(T_{d+1},\Delta_d) & \partial(B_{d+1},\Delta_d)\\ \partial(T_{d+1},B_d) & \partial(B_{d+1},B_d) \end{pmatrix}.$$
 
 Since $b\notin X$ for any $X\in T_{d+1}$, so $\partial(T_{d+1},B_d)=0$. A similar argument shows that $\partial(B_{d+1},T_d)=0$. 
 
 Next, notice that since $t\in X$ for all $X\in T_{d+1}$ we may find a $X'\in\Delta_d$ such that $X=X'\cup\{t\}$. Thus there is exactly one way for $\partial X\in\Delta(d)$ and that is to delete $t$ so that the image is $X'$. Each $X'$ corresponds to a row of $\partial(T_{d+1},\Delta_d)$ and has only one nonzero entry corresponding to the column for $X'\cup\{t\}$. Since $t$ is the 0 position of this simplex, it must be that the nonzero entry is $-1$. Thus there is some order so that $\partial(T_{d+1},\Delta_d)=-I$. A similar argument shows that $\partial(B_{d+1},\Delta_d)=I$.
 
 Finally, we notice that for $X\in T_{d+1}$ the only way for $\partial X\in T_d$ is if we do not delete $t$. Thus we may ignore $t$ and see that $\partial(X,T_d)$ is determined by $X'$, but since $t$ occupies the 0 position, all the signs are altered and $\partial(X,T_d)=-\partial(X')$. Thus $\partial(T_{d+1},T_d)=-\partial(\Delta)$. A similar argument shows that $\partial(B_{d+1},B_d)=\partial(\Delta)$.
 
 $\partial(\Sigma\Delta)$ now has the following form. From this, it is entirely clear that $\partial(\Delta)\varphi^T \equiv 0 \mod q$ and $\partial(\Sigma\Delta)(\alpha,\beta)^T\equiv 0 \mod q$ if and only if $\varphi=\alpha=\beta$. 
 
$$\begin{pmatrix} -\partial(\Delta) & 0\\ -I & I \\ 0 & \partial(\Delta)\end{pmatrix}.$$
 \end{proof}
 
 Since $\Sigma\Delta$ has dimensions $d+1$, the proof of Proposition \ref{weakincrease} follows immediately.

\subsection{The Complete Complex}

In this section we endeavor to prove Theorem \ref{lower}. The motivation for phrasing this theorem was Tutte's original work in \cite{Tu} wherein he proved that the Petersen graph has no 4-flows. On first considering the generalized case, I wrote a brute force program to compute the flow quasipolynomial for hundreds of random simplicial complexes. In the end, a clear pattern emerged in simplicial complexes of $K_n^{n-2}$.

\begin{prop}\label{completeflows}
$$\Phi_{K_n^{n-2}}(q)=\prod_{i=1}^{n-1}(q-i).$$
\end{prop}

To prove that these complexes have this simple form we state and prove a few lemmas that apply to all $K_n^k$ for $2\le k\le n-1$. 

\begin{lemma}[Cordovil \& Lindstr\"{o}m \cite{CL}]\label{completeform} Let $2\le k \le n-1$. Then
$$\partial(K_n^k)=\begin{pmatrix} -\partial(K_{n-1}^{k-1}) & 0\\ I & \partial(K_{n-1}^k)\end{pmatrix}.$$
\end{lemma}

\begin{proof}
First note the $K_n^{k-1}$ is the codimension 1 skeleton of $K_n^k$. Now we begin by partitioning the two complexes. Let $K_n^k=(X_0,X)$ and let $K_n^{k-1}=(Y_0,Y)$ where $f\in X_0$ and $r\in Y_0$ if and only if $0\in f$ and $0\in r$. Of course, it is clear that $\partial(K_n^k)$ has the following form.

$$\partial(K_n^k)=\begin{pmatrix} \partial(X_0,Y_0) & \partial(X,Y_0)\\ \partial(X_0,Y)& \partial(X,Y)\end{pmatrix}.$$

First, we remark that $X$ is defined to be those facets not containing $0$, so no member of the image $\partial(X)$ can contain 0. Thus $\partial(X,Y_0)=0$.

Next, we remark that for $f\in X_0$ to have its image in $Y$ one must delete $0$. Thus every column of $\partial(X_0,Y)$ has exactly one nonzero entry, which is by definition either 1 or $-1$. Moreover each row must have exactly one entry since deleting $0$ will give unique ridges. Thus we may choose an order of $X_0$ so that $\partial(X_0,Y)=I$. We will use the fact that $\partial(X,Y_0)=0$ to give us the freedom to specify orders for $Y_0$, $Y$, and $X$.

Now, we note that for $f\in X_0$ and $r\in Y_0$ that there are unique $f'\in X_0\setminus\{0\}$ and $r'\in Y_0\setminus\{0\}$. Define the map $d$ by $d:f\mapsto f'$ and $d:r\mapsto r'$ for each $f\in X_0$ and $r\in Y_0$. It is clear that $\partial(d(f))=d(\partial(f))$. Now note that $\partial\circ d: K_{n-1}^{k-1}\rightarrow K_{n-1}^{k-2}$. Since $d: X_0\rightarrow K_{n-1}^{k-1}$ and $d: Y_0\rightarrow K_{n-1}^{k-2}$ are set isomorphisms that commute with $\partial$, it must be that for some order of $Y_0$ $\partial(X_0,Y_0)=\partial\circ d(X_0,Y_0)=-\partial(K_{n-1}^{k-1}).$ Where the negative sign is needed to take into account that $d$ shifts every element of a simplex down 1. 

Finally, note that each $f\in X$ and $r\in Y$ we may consider as members of $K_{n-1}^k$ and $K_{n-1}^{k-1}$ respectively simply by ignoring $0$ in $K_n^k$ and $K_n^{k-1}$. Since $0$ does not appear in any such $f$ or $r$ the signs in the boundary map are not changed. So for some choice of order for $Y$ we must have $\partial(X,Y)=\partial(K_{n-1}^k)$. 

\end{proof}

\begin{lemma}[Cordovil \& Lindstr\"{o}m \cite{CL}] The rank of $\partial(K_n^k)$, i.e. $r(M(K_n^k))$, is $\binom{n-1}{k-1}$.
\end{lemma}
\begin{proof} Let $X_0$ and $X$ be as in the proof of Lemma \ref{completeform}. Since deleting $0$ for each $f\in X_0$ yields a unique $r\in Y$, each row of $\partial(X_0)$ corresponding to the $r\in Y$ has a unique nonzero elements. Thus $\partial(X_0)$ is a set of linearly independent column vectors. I claim this is a basis.

Let $z\in X$. Then $X_0\cup z$ contains every subset of size $k$ of $z\cup\{0\}$. This means there is a subset of $X\cup z$ isomorphic to $K_{k+1}^{k}$, which is the boundary of some $l\in K_{n+1}^k$. Thus the kernel of $\partial(X\cup z)$ contains the kernel of $\partial(\partial(l))$, which contains $\partial l$, since $\partial^2 =0$. So $X\cup z$ is a dependent set, and this is true for all $z\in X$ thus $X_0$ is a maximal independent set. And since $X_0\cong K_{n-1}^{k-1}$, $\text{rank}\partial(K_n^k)=|X_0|=\binom{k-1}{n-1}$.

\end{proof}

Note that the above two arguments depended entirely on the combinatorics of $K_n^k$ and not on the characteristic of the field over which $\partial$ is defined. This allows the following argument to remain valid over all fields.

\begin{lemma}\label{rowreduction} If $A=[I|\partial(K_{n-1}^k)]$, then $M(\partial(K_n^k))\cong M(A)$, and
$$\Phi_{K_n^k}(q) = |\{\varphi\in \text{\emph{ker}} A \mod q:\varphi\text{ is nowhere zero}\}|.$$
\end{lemma}

\begin{proof}
Since the rank of $\partial(K_n^k)$ is $\binom{k-1}{n-1}$, the rows of $\partial(K_n^k)$ in Lemma \ref{completeform} present in $A$ are clearly independent, and the number of rows in $A$ is exactly $|K_{n-1}^{k-1}|=\binom{k-1}{n-1}$, the rows of $A$ constituted a row basis of $\partial(K_n^k)$. So by elementary row operations, $\partial(K_n^k)$ is equivalent to the following form. The claims follow immediately. 
$$\begin{pmatrix}  0\\ A\end{pmatrix}.$$
\end{proof}

Everything up to this point has been valid for any $2\le k\le n-2$. We now assume that $k=n-2$

\begin{lemma}\label{countingflows} Let $k=n-2$ and $A$ be as in Lemma \ref{rowreduction}, $\varphi$ is a nowhere zero vector in $\text{\emph{ker}}(A) \mod q$ if and only if $\varphi_j\not\equiv \varphi_j$ and $\varphi_j\not\equiv 0$ for $\binom{n-1}{2} < i\le j\le \binom{n}{2}$. \end{lemma}

\begin{proof}
To do this we first note that there is a bijection between solutions of $A\varphi^T\equiv 0$ and solutions of $\partial(K_{n-1}^{n-2})\psi^T\not\equiv 0$. To see this, for each such $\varphi$, set $\psi$ equal to the last $n-1$ entries of $\varphi$ and we see that $\partial(K_{n-1}^{n-2})\psi^T$ equals the first $\binom{n-1}{2}$ of $\varphi$ and therefore is nowhere equivalent to zero.  And for each such $\psi$, let $\varphi=(\partial(K_{n-1}^{n-2})\psi,-\psi)$, then $A\varphi^T\equiv 0$.

Now to complete the claim we must invoke the structure of $K_{n-1}^{n-2}$. Since it is the boundary of the $(n-2)$-simplex $K_{n-1}^{n-1}$, homeomorphic to an $(n-2)$-ball, $K_{n-1}^{n-2}$ is a triangulation of an $(n-3)$-sphere. Moreover, it is clear that since $S^{n-3}$ is a manifold that every ridge is only incident with two facets. So every row of $\partial(K_{n-1}^{n-2})$ has exactly two nonzero entries. And since we know $S^{n-3}$ is $\mathbb{Z}$-orinentable, there is a choice of orientation so that each pair of columns has exactly one common row with opposite signs. Assuming we are solving $\partial(K_{n-1}^{n-2})\psi^T\not\equiv 0$, the set of equations is $\psi_i\not\equiv \psi_j$ for $i\neq j$. This can be reduced to the set of equations in the claim. 
\end{proof}

\begin{proof}[Proof of Proposition \ref{completeflows}] By Lemma \ref{rowreduction}, $$\Phi_{K_{n-1}^{n-2}}(q)=|\{\varphi\in\text{ker}A\mod q:\varphi\text{ is nowhere zero}\}|.$$ By Lemma \ref{countingflows}, we can count such $\varphi$. Let $\psi$ be a nowhere-zero solution to $\partial(K_{n-1}^{n-2})\psi\not\equiv 0$, then we may choose any $\psi_1\in\{1,\dots,q-1\}$. Suppose we have fixed choices for $\psi_1,\dots,\psi_j$, then we may choose any $\psi_{j+1}\in \{1,\dots,q-1\}\setminus \{\psi_1,\dots,\psi_j\}$. And since $|K_{n-1}^{n-2}|=n-1$, we can see that the total number of choices is $(q-1)\dots(q-(n-1))$, which proves the claim.
\end{proof}

\begin{proof}[Proof of Theorem \ref{lower}] By Proposition \ref{completeflows},
$$\Phi_{K_{d+3}^{d+1}}(q)= \prod_{i=1}^{d+2}(q-i).$$
$K_{d+3}^{d+1}$ is a simplicial complex of dimension $d$ and has no $q$-flows for $q\le d+2$. Thus $$\kappa(d)>d+2.$$
\end{proof}

While this is a weaker result than Tutte's, it holds in all dimensions. Were there a consistent way to construct a ``Petersen complex'' that generalizes all the necessary properties of the Petersen graph, we might be able to prove that $\kappa(d)>d+3$. However, the established patterns and manageable construction of the complete complex makes an easy improvement unlikely. 
 
\section{The Upper Bound}\label{theupperbound}

As already phrased in Theorem \ref{upper} we do not have a proof that $\kappa(d)$ is finite. Instead we need to assume that a complex is highly facet-connected. This may seem as a copout, however the proof we are adapting from graph theory has the same assumption. Jaeger assumed that a graph was $3$-edge-connected, which is $d+2$ for graphs, to demonstrate the existence of a $2^3$-flow. He concluded by showing that every $2$-edge-connected graph is comprised of a connected sum of two $3$-edge-connected graphs in a way that allows the construction of a $2^3$-flow on the original graph. Luckily for Jaeger, a $2$-edge-connected graph is a bridgeless graph, and his proof was complete. This section we provide the details for a generalization of his argument \emph{except} for the reduction from $(d+2)$-facet-connected to all bridgeless complexes.

As promised we will need to use $G$-flows, which we will now neatly define in terms of homology.

\begin{definition}
Let $G$ be a finitely generated abelian group, then a $G$-flow on $\Delta$ is a vector in $H_d(\Delta; G)$. A $G$-flow is nowhere-zero if no entry is the identity of $G$. 
\end{definition}

For graphs, Tutte \cite{Tu} showed that the the number of nowhere-zero $G$-flows depends only on $|G|$, and so for graphs $q$-flows classify all possible $G$-flows. Tutte's proof of this fact follows from the contraction-deletion identity of $\Phi_G(q)$, and indeed when $\Phi_\Delta(q)$ is a polynomial the same contraction deletion identity holds and the argument follows. We recover a slight relationship in the case of elementary abelian group of exponent 2.

\begin{lemma}\label{elementaryabelian} If there is a nowhere-zero $\mathbb{Z}_2^r$-flow of $\Delta$, then there is a nowhere-zero $2^r$-flow of $\Delta$.
\end{lemma}

\begin{proof}
For the purpose of this proof, let $n=|F|$. Let $\varphi=(\varphi_1,\dots,\varphi_n)$, where $\varphi_i=(\varphi_{i1},\dots,\varphi_{ir})\in\mathbb{Z}_2^r$. And suppose that $\partial\varphi^T\equiv 0\mod 2$. So it must then be that for every row $j$ of $\partial$ and every $k$ component of $\varphi_i$ that
$$\sum_{i=1}^n\partial_{ij}\varphi_{ik}\equiv 0\mod 2.$$

Thus for each $k$ here must be an even number of nonzero $\varphi_{ik}$, thus we may consider a new $\varphi_{ik}'\in\{-1,0,1\}$ so that the sum be taken in $\mathbb{Z}$:
$$\sum_{i=1}^n\partial_{ij}\varphi_{ik}'=0.$$
So clearly $\varphi'=(\varphi_1',\dots,\varphi_n')$ is a $\{-1,0,1\}^r$-flow. 

Now we endeavor to define a $2^r$-flow $y=(y_1,\dots,y_n)$.
$$y_i=\sum_{k=1}^r\varphi_{ik}'2^{k-1} \mod 2^r.$$
Now for every row $j$ we have.
$$\sum_{i=1}^n\partial_{ij}y_i=\sum_{i=1}^n\partial_{ij}\sum_{k=1}^r \varphi_{ik}'2^{k-1}=\sum_{k=1}^r 2^{k-1}\sum_{i=1}^n\partial_{ij}\varphi_{ik}'.$$
But of course since the rightmost sum is zero in $\mathbb{Z}$, it must be equivalent to zero $\mod 2^r$. And so $y$ is a $2^r$-flow.

Finally we must check that if $y_i\equiv 0\mod 2^r$ then $\varphi_i\equiv 0\in\mathbb{Z}_2^r$. From $y_i$'s binary definition, it is clear that $y_i\equiv 0 \mod 2^r$ only if $y_i =0$. Let us assume that $y$ is not nowhere-zero, assume that $y_i=0$. Now define $P$ to be the set of numbers $k$ for which $\varphi_{ik}'=+1$ and define $N$ to be the set of numbers $k$ for which $\varphi_{ik}'=-1$. 
$$y_i=\sum_{i\in P} 2^i - \sum_{j\in N} 2^j .$$
Now assume that $y_i=0$ for some $i$. Then
$$\sum_{i\in P} 2^i = \sum_{j\in N} 2^j.$$
Since $P$ and $N$ are disjoint and since binary expansions are unique it must be that $P=N=\emptyset$. Thus $\varphi_{ik}=0$ for all $k$, and so $\varphi_i\equiv 0\in \mathbb{Z}_2^r$.
\end{proof}

Of course, there need not be the same number of $\mathbb{Z}_2^r$-flows as $2^r$-flows as seen in Section \ref{background}. But as Tutte was only interested in the existence of a single flow, so shall we. 

To proceed unto Jaeger's argument we must state a curious bit of terminology from graph theory. In graph theory the \emph{arboricity} of a graph is the minimal number of trees needed to cover the graph. And of course \emph{coarboricity} is exactly as you expect. 

\begin{definition} The arboricity of $\Delta$ is the least number $a$ for which there are $a$ forests covering $\Delta$, and the coarboricity is the least number $c$ for which there are $c$ coforests covering $\Delta$.
\end{definition}

We will use a coforest covering to construct an explicit $\mathbb{Z}_2^c$-flow.

\begin{lemma}\label{constructedflow} If $\Delta$ has coarboricity $c$, then $\Delta$ has a nowhere-zero $2^c$-flow.
\end{lemma}
\begin{proof}
Suppose $\Delta = \bigcup_{i=1}^c B_i$, where $B_i$ is a coforest. So $B_i$ is a cobase of $M(\Delta)$, and so $\Delta\setminus B_i$ is a basis. From matroid theory every basis and disjoint element uniquely define a circuit contained in their union called the fundamental circuit. For an element $f\in B_i$, define $C_f^i$ to be the fundamental circuit of the basis $\Delta\setminus B_i$ and element $f$. Now we define an indicator function of this circuit.

$$\varphi_i^f(e)=\begin{cases}1 & \text{if }e\in C_f\\ 0 & \text{ else.} \end{cases}$$
And now take the sum mod 2 over all elements of this cobasis.
$$\varphi_i = \sum_{f\in B_i} \varphi_i^f \mod 2.$$
When you consider $\varphi_i$ as a vector in $\mathbb{Z}_2^{|F|}$ it is clearly a $\mathbb{Z}_2$-flow, since it is the sum of circuits, i.e. elements of the kernel of $\partial$. Note that for $f,f'\in B_i$ that $\varphi_i^f(f')=1$ if and only if $f'=f$. So the function $\varphi_i$ is nowhere-zero in $B_i$. Now we define a $\mathbb{Z}_2^c$-flow.
$$\varphi=\bigoplus_{i=1}^c \varphi_i.$$
Since for every facet $f\in \Delta$ there is a coforest $B_i$ containing $f$, $\varphi_i(f)=1$ and so $\varphi(f)\not\equiv 0\in \mathbb{Z}_2^c$. Thus $\varphi$ is a nowhere-zero $\mathbb{Z}_2^c$-flow.

By Lemma \ref{elementaryabelian}, $\Delta$ has a $2^c$-flow.
\end{proof}

Though perhaps it should have been immediately apparent from the definition of bridgeless, the fact every bridgeless complex has finite coarboricity gives as a corollary of the theorem a concrete proof to the following fact.

\begin{corollary} If $\Delta$ is bridgeless then $\Phi_\Delta(q)\neq 0$ for some $q$.
\end{corollary}

\begin{proof}
$\Delta$ has no coloops, so every facet constitutes a coindependent set. The third axiom of independent sets, allows you extend every coindependent set to a cobase. Thus every facet belongs to a cobase, and this set of cobases trivially covers $\Delta$. Thus $\Delta$ has finite coarboricity, and Lemma \ref{constructedflow} proves the claim.
\end{proof}

Even though every complex has finite coarboricity, this does not imply that $\kappa(d)$ is finite. Consider a triangulation of the sphere $S^2$ with $n$ facets. One may obtain a spanning tree by deleting a single facet, thus the cotrees are comprised of a single facet. And so the coarboricity must be $n$. Clearly we may let $n$ go to infinity. But of course, $S^2$ has a $2$-flow. 

Coarboricity is a somewhat contrived constant, and indeed Jaeger used a theorem of Edmonds to translate the coarboricity requirement into terms of edge-connectedness, which is much closer to the essential property for matroids, connectivity. We state this theorem of Edmonds, and exploit it to the same ends. 

\begin{prop}[Edmonds \cite{Ed}] A matroid $M=(E,r)$ is the union of $c$ independent sets if and only if $cr(X)\ge |X|$ for all $X\subseteq E$.
\end{prop}

Now we set about reducing coarboricity to facet-connectedness. This will allow us to prove the main statement of the upper bound of $\kappa(d)$, which was given with hypotheses on the facet-connectedness of $\Delta$. Note that we may have just as easily used a hypothesis that $\Delta$ have finite coarboricity, but we carry out the following to complete the analogy to Jaeger.

\begin{lemma}\label{coarboricity} If $\Delta$ is $(d+2)$-facet-connected than $\Delta$ has coarboricity at most $d+2$.
\end{lemma}
\begin{proof} For ease of notation let $b=\beta_{d-1}(\Delta\setminus X)$. Clearly there exists $b$ distinct face cuts $F_i\subseteq X$. These face cuts possibly overlap, but clearly $\bigcup_{i=1}^b F_i\subseteq X$. Because each facet of a face cut has $d+1$ ridges, it can contribute to at most $(d+1)$ of the face cuts $F_i$. Thus
$$\sum_{i=1}^b |F_i|\le (d+1)\left|\bigcup_{i=1}^b F_i \right| \le (d+1)|X|.$$
Moreover, since $\Delta$ is $(d+2)$-facet-connected, $d+2\le |F_i|$ for all $i$, and so 
$$b(d+2)\le \sum_{i=1}^b|F_i|\le (d+1)|X|.$$

Now we will apply Edmond's Proposition to $M^*(\Delta)$. Recall that $r^*(X)=|X|+\beta_{d-1}(\Delta)-b$. So this with the above yields
$$r^*(X)\ge |X| - b \ge |X|- \frac{d+1}{d+2}|X| = \frac{1}{d+2}|X|.$$
Thus for all $X$ we have $(d+2)r^*(X)\ge |X|$, and so there is a $(d+2)$ coindependent set covering of $M(\Delta)$. Extend each coindependent set to a cobase, and so the coarboricity of $\Delta$ is at most $d+2$.
\end{proof}

We now have the necessary facts to prove Theorem \ref{upper}.
\begin{proof}[Proof of Theorem \ref{upper}] By Lemma \ref{coarboricity}, $\Delta$ has coarboricity most $c\le d+2$. By Lemma \ref{constructedflow}, $\Delta$ has a $2^c$-flow.
\end{proof}

This concludes our tracing of Jaeger's argument, but we make one final conjecture. The claim that $\kappa(d)< \infty$ can be reduced to proving it.

\begin{conj} Fix $\kappa$. Let $d+2\ge k>2$. If for all $k$-facet-connected $\Delta$ there is a $q$-flow for some $q<\kappa$, then all $(k-1)$-facet-connected $\Delta$ also have a nowhere-zero $q$-flow for some $q<\kappa$.
\end{conj}

\section{Conclusions and Acknowledgements}

It is perhaps condemnable to add generality to conjecture without making any further progress on the original. Indeed, Tutte was an undergraduate when he began discovering his famous results about flows and colorings, so why should not I (an undergraduate when this work was completed), have been able to solve his conjecture. Regardless, I have a good feeling that the results of Theorem \ref{planardual} might be used to shed new light to Tutte's original conjecture, as surely every graph is dual to some 2-dimensional complex inside $S^3$.

It is possible that my amateur status in matroid theory has kept me from being able to fully answer my own questions as there may be an abstract construction in matroid theory analogous the trick that Jaeger used to prove his reduction lemma. 

It is also important to remember that the existence of a $q$-flow of graphs implies the existence of $(q+1)$-flow. This is absolutely false for $q$-flows on a simlicial complex. So perhaps asking what $\kappa(d)$ is not as meaningful as it was for graphs.

The author would like to thank Noah Taylor for asking me to justify myself rather than run on pure conjecture as is my nature. He would also like to thank Noah for being infinitely better at calculations both by hand and by computer. The author would like to thank Dr. Igor Pak for some references. The author would also like to thank Dr. Jeremy Martin for correspondence and kind words. The author thanks the anonymous referees of the 2014 DMTCS proceedings of FPSAC, focused my attention and allowed this paper to emerge from my otherwise unintelligible notes. Of course, the author is indebted to The Ohio State University (especially the College of Arts and Sciences) for funding  while doing the research over the summer of 2013 and for funding travel to various conferences that school year. Finally, the author owes the entire possibility to do this research to Dr. Sergei Chmutov, without whom there would be little direction or resources for young mathematicians at OSU. And finally, for Tutte, Jaeger, Seymour, Oxley, and all our mathematical forbearers for setting down the amazing realities of mathematics they saw before them so that posterity might strive further, thank you.

\end{document}